\documentclass[11pt]{amsart}

\usepackage{amsmath}
\usepackage{amssymb,amsfonts,mathrsfs, amsthm, stmaryrd, color}
\usepackage[margin=2 cm,heightrounded=true,centering]{geometry}

\usepackage{amscd}
\usepackage{verbatim}
\usepackage{euscript}

\newtheorem{theorem}{Theorem}[section]
\newtheorem{proposition}[theorem]{Proposition}
\newtheorem{lemma}[theorem]{Lemma}
\newtheorem{remark}[theorem]{Remark}

\newtheorem{definition}[theorem]{Definition}
\newtheorem{conjecture}[theorem]{Conjecture}

\DeclareMathOperator{\ric}{Ric}
\DeclareMathOperator{\hess}{Hess}
\DeclareMathOperator{\tr}{tr}
\DeclareMathOperator{\SO}{SO}
\DeclareMathOperator{\so}{so}
\DeclareMathOperator{\diag}{diag}

\begin{document}

\title[The rigid Horowitz-Myers conjecture]{The rigid Horowitz-Myers conjecture}

\author{Eric Woolgar}
\address{Dept of Mathematical and Statistical Sciences, and Theoretical Physics Institute, University of Alberta, Edmonton, AB, Canada T6G 2G1.}

\email{ewoolgar(at)ualberta.ca}

\date{\today}

\begin{abstract}
\noindent The \emph{new positive energy conjecture} was first formulated by Horowitz and Myers in the late 1990s to probe for a possible extended, nonsupersymmetric AdS/CFT correspondence. We consider a version formulated for complete, asymptotically Poincar\'e-Einstein Riemannian metrics $(M,g)$ with bounded scalar curvature $R\ge -n(n-1)$ and with no (inner) boundary except possibly a finite union of compact, totally geodesic hypersurfaces (horizons). This version then asserts that any such $(M,g)$ must have mass not less than a certain bound which is realized as the mass $m_0$ of a metric $g_0$ induced on a time-symmetric slice of a spacetime called an AdS soliton. This conjecture remains unproved, having so far resisted standard techniques. Little is known other than that the conjecture is true for metrics which are sufficiently small perturbations of $g_0$. We pose another test for the conjecture. We assume its validity and attempt to prove as a corollary the corresponding scalar curvature rigidity statement, which is that $g_0$ is the unique asymptotically Poincar\'e-Einstein metric with mass $m_0$ obeying $R\ge -n(n-1)$. Were a second such metric $g_1$ not isometric to $g_0$ to exist, it then may well admit perturbations of lower mass, contradicting the assumed validity of the conjecture. We find enough rigidity to show that the minimum mass metric must be static Einstein, so the problem is reduced to that of static uniqueness. When $n=3$ the manifold must be isometric to a time-symmetric slice of an AdS soliton spacetime, or must have a non-compact horizon. En route we study the mass aspect, obtaining and generalizing known results: (i) we relate the mass aspect of static metrics to the holographic energy density, (ii) we obtain the conformal invariance of the mass aspect when the bulk dimension is odd, and (iii) we show the vanishing of the mass aspect for negative Einstein manifolds with Einstein conformal boundary.
\end{abstract}

\maketitle

\section{Introduction}
\setcounter{equation}{0}

\noindent The AdS/CFT correspondence animated theoretical physics at the turn of this century by holding out the hope that strongly coupled quantum field theories could be studied simply by examining supergravity theory in the classical limit \cite{Maldacena, Witten3, HS}. This is simultaneously much more than one could have hoped and less than one might ideally have wished. The correspondence applies only between certain supergravities in the bulk asymptotically anti-de Sitter (AdS) spacetime and specific supersymmetric quantum field theories on a compact manifold of one less dimension, where the compact manifold is the conformal boundary for the bulk manifold. Physicists would like to study strongly coupled gauge theories that appear to describe the world we experience, and which therefore are not supersymmetric, such as quantum chromodynamics or QCD. Now QCD with massive quarks is not conformal either, but even an AdS/CFT correspondence for QCD with massless quarks would be very desirable, so a non-supersymmetric AdS/CFT correspondence would be of great interest.

Horowitz and Myers \cite{HM} considered possible consequences of such a desired non-supersymmetric AdS/CFT correspondence, were such a thing to exist. They studied a remarkable solution of the Einstein equations which they called an \emph{AdS soliton}. It is globally static and asymptotically locally anti-de Sitter. We will refer to the time-symmetric slices of an AdS soliton as \emph{Horowitz-Myers geons}. These time-symmetric slices are conformally compactifiable and asymptotically hyperbolic---indeed, they are \emph{asymptotically Poincar\'e-Einstein} (APE) \cite{BMW1}---and admit a mass according to the definition of Wang \cite{Wang} as generalized by Chru\'sciel and Herzlich \cite{CH}. Surprisingly, the mass of a Horowitz-Myers geon is \emph{negative}.

Positive mass theorems exist for asymptotically anti-de Sitter spacetimes \cite{AD, GHW, Woolgar} and for conformally compactifiable, asymptotically hyperbolic Riemannian manifolds \cite{MinOo, Wang, CH, ACG}. However, most of these theorems assume either a suitable spinor structure admitting global, asymptotically constant solutions of the Witten equation, or at least a spherical conformal boundary at infinity. None of them apply, and none of them have been successfully adapted, in sufficient generality to manifolds with toroidal conformal infinity (but see \cite{LN}). The Horowitz-Myers geons have toroidal conformal infinity and, while they admit spinor structure, they do not admit asymptotically constant Killing spinors. Thus they cannot be supersymmetric and the arguments of the Witten positive energy proof do not apply.

Horowitz and Myers estimated the ground state Casimir energy of a conformal field theory on a flat torus and compared it to the mass of the Horowitz-Myers geon for which this torus could serve as the conformal boundary at infinity. They found that these matched, modulo a factor of $3/4$ which was not unexpected due to the estimation method for the ground state energy calculation (which was estimated at weak coupling). They then formulated a series of conjectures, based on this matching and the understanding of the Casimir energy as the CFT ground state energy, to the effect that there should be what they called a \emph{new positive energy theorem} for spacetimes with the same conformal boundary at infinity as the AdS solitons (and with suitable asymptotic behaviour on approach to the conformal boundary). The conjectures vary according to the context in which they are set, but they each posit that the infimum of the masses of all spacetimes with the conformal boundary of the soliton, suitable asymptotics, and an appropriate lower bound governing bulk curvature (i.e., a pointwise \emph{energy condition}) would be realized by a Horowitz-Myers geon; i.e., a time-symmetric slice of an AdS soliton. The veracity of the conjectures would be evidence in favour of a non-supersymmetric version of AdS/CFT. We focus here on one of their conjectures in particular.

\begin{conjecture}[{{Riemannian Horowitz-Myers Conjecture; \cite[Conjecture 4.3]{HM}, \cite{CM}}}] \label{conjecture1.1}
Let $(M,g)$ be a complete asymptotically Poincar\'e-Einstein (APE) $n$-manifold $n\ge 3$, with compact, totally geodesic, possibly empty boundary and flat toroidal conformal infinity. Let the scalar curvature $R_g$ of $g$ obey $R_g+n(n-1)\ge 0$. Then the Wang mass-energy $m$ of $(M,g)$ obeys $m\ge m_0$ where $m_0$ is the mass of the Horowitz-Myers geon which has least mass amongst all Horowitz-Myers geons with the same conformal boundary-at-infinity.
\end{conjecture}

We have modified this conjecture from the original in various ways. First, the original was posed only for $n=4$ because it was motivated by string phenomenology considerations on a ten-dimensional warped product spacetime. Kaluza-Klein reduction applied to the fibres of the warped product reduced the ten dimensions to five, and then a time-symmetric slice was taken to get to a $4$-dimensional Riemannian manifold. This motivation notwithstanding, the conjecture can be posed in any dimension $n\ge 3$, and indeed this was done in \cite{CM}. Second, the motivation made use of spacetime and so the conjecture spoke of spacetime AdS asymptotics, but since the version of interest here is Riemannian, it is natural to phrase it in Riemannian terms. In that case, the natural choice is to use APE asymptotics \cite{BMW1}; see Definition \ref{definition2.1}. Third, the original conjecture made no mention of (inner) boundaries. We include compact totally geodesic boundaries to allow for ``horizons''. Fourth, the original conjecture was posed before the notion of the mass of an asymptotically hyperbolic manifold was made clear by Wang (\cite{Wang}, see also \cite{CH}), so it was phrased in terms of a limit of a certain quasi-local mass. As there is now a good definition of asymptotically hyperbolic mass, we have updated the phrasing to use this definition. And fifth, there is a countable infinity of non-isometric Horowitz-Myers geons which share the same conformal boundary (\cite{Anderson}, \cite{GSW2}). The one with least mass is the one chosen so that the shortest nontrivial cycle on the boundary at infinity bounds a disk in the bulk \cite{Page}. It is the mass of this geon which serves as $m_0$.

If true, the conjecture would provide evidence in favour of an AdS/CFT correspondence in the absence of supersymmetry, and perhaps allow us to understand the confining phase of QCD in terms of general relativity in asymptotically AdS spacetimes \cite{Witten3, HP}. However, the conjecture remains unproved. Indeed, since the conjecture was first posed, there has been very little progress. The main point of the conjecture is that it applies to manifolds where the Witten spinorial method cannot be applied. Other methods, such as those of Schoen-Yau \cite{SY} or the inverse mean curvature flow technique first discussed by Geroch \cite{Geroch, HI} fail also \cite{Gibbons}. One difficulty is that these techniques tend to rely (implicitly) on various forms of comparison to the putative minimal mass metric (ground state). These comparisons work well when the ground state has constant curvature, but that is not the case here.

Herein we prove a corollary of Conjecture \ref{conjecture1.1}.

\begin{theorem}[Scalar curvature rigidity of mass-minimizing Horowitz-Myers geons]\label{theorem1.2}
Assume that Conjecture \ref{conjecture1.1} is true. Let $(M,g)$ be a complete, asymptotically Poincar\'e-Einstein (APE) $n$-manifold with toroidal conformal infinity and with Wang mass $m=m_0$, where $m_0$ is the mass of the Horowitz-Myers geon $(M_0,g_0)$ which has least mass among all Horowitz-Myers geons with that conformal infinity. Let the scalar curvature $R_g$ of $(M,g)$ obey $R_g+n(n-1)\ge 0$.
\begin{enumerate}
\item[(i)] Then $(M,g)$ is an APE static Einstein metric.
\item[(ii)] If $n=3$ and the static potential $N$ obeys $xN\to 1$ at conformal infinity for $x$ a special defining function (i.e., $|dx|_{x^2g}=1$ on a collar of conformal infinity), then $(M,g)$ is isometric to $(M_0,g_0)$.
\end{enumerate}
\end{theorem}

Conclusion (ii) of Theorem \ref{theorem1.2} says that the unique static, asymptotically locally anti-de Sitter 4-dimensional spacetime that one constructs from $(M,g)$ by solving the static Einstein equations (see Section 2.3) is an AdS soliton (see Section 2.4).\footnote
{We take this opportunity to emphasize that in the present paper, $n$ will denote the dimension of the Riemannian manifold which serves as a time-symmetric slice of an $(n+1)$-dimensional spacetime.}
Anderson \cite{Anderson} proved that $4$-dimensional AdS solitons are the unique asymptotically AdS, complete, static Einstein spacetimes with toroidal boundary. This is the $n=3$ case in our terminology. The Riemannian result in \cite{Anderson} is expressed in terms of static spacetimes in \cite{ACD}. The condition on the zero set of the static potential excludes static configurations of ``black branes'' that extend to infinity in such a way that $(M,g)$ remains APE. Such configurations, if they can exist (see \cite{GM} for results in the asymptotically flat case), would not be slices of asymptotically anti-de Sitter spacetimes.

To understand this result as a test of Conjecture \ref{conjecture1.1}, recall the analogous case of nonsupersymmetric Kaluza-Klein asymptotically flat spacetimes. There, a second zero-mass metric, not isometric to a Kaluza-Klein manifold with flat base, was found \cite{Witten2}. Perturbations of this metric were then found to have (unbounded) negative mass \cite{BP}, destabilizing the theory. The above result is slightly stronger. For $n=3$, a second metric (not isometric to $g_0$ but otherwise obeying the above conditions) with mass $m_0$ would immediately falsify Conjecture \ref{conjecture1.1} without having to test for negative mass perturbations. In higher dimensions, the above theorem does not rule out a second such metric, but it must be static Einstein (the example in \cite{Witten2} is not static, and perhaps this may stabilize such a metric against negative mass perturbations).

Little can currently be said in higher dimensions beyond that any such mass minimizing metric must be static Einstein. The issue is the absence of uniqueness theorems for static Einstein metrics. A uniqueness theorem for the AdS soliton, and thus for Horowitz-Myers geons, is available (\cite{GSW1}, \cite{GSW2}), but it requires assumptions on the asymptotics that are stronger than those needed in the $n=3$ case. Moreover, static black hole uniqueness theorems would also be helpful, but uniqueness of higher dimensional static anti-de Sitter black holes remains an open question. Nonetheless, it is useful to say something because, in doing so, we illuminate mechanisms by which the Horowitz-Myers conjecture might fail in higher dimensions. We postpone this discussion to the concluding section, where we state limited results.

Finally, we note that it is not unusual for the rigidity case to be studied before the corresponding positive energy theorem is available (in the present case, we of course replace ``positive'' by ``bounded below''). A similar, though not entirely analogous thing occurred for asymptotically hyperbolic manifolds with spherical conformal infinity, where Min-Oo \cite{MinOo} proved a scalar curvature rigidity theorem for zero-mass spin manifolds with spherical conformal infinity. Andersson and Dahl \cite{AnderssonDahl} were able to generalize this result to manifolds whose conformal infinity was a quotient of a sphere.

This paper is organized as follows. In Section 2 we give some background. Asymptotically hyperbolic manifolds and related concepts are explained in Section 2.1. We define the Wang mass for certain asymptotically hyperbolic manifolds in Section 2.2. We study its conformal invariance properties. Proofs are relegated to the Appendix. Static Einstein metrics with negative scalar curvature are discussed in Section 2.3. This enables us to relate the Wang mass aspect to the vacuum expectation values of the stress-energy tensor for the boundary conformal field theory of the AdS/CFT correspondence. In Section 2.4, we introduce AdS solitons and their time-symmetric slices, the Horowitz-Myers geons. In section 2.5 we explain why standard techniques for proving rigidity do not work in this situation. In Section 3.1 we use work of Corvino \cite{Corvino} to show that when the ``static Einstein operator'' (the formal adjoint of the linearized scalar curvature operator) has trivial kernel, there is a variation of the metric which increases the scalar curvature and does not change the mass. In section 3.2, we use a solution of the Yamabe problem to construct another metric which has constant scalar curvature and strictly lower mass than the initial metric. This violates the Horowitz-Myers conjecture. In Section 3.3 we therefore conclude that the kernel of the static Einstein operator must have been nontrivial, and the original metric must have paired with a lapse function which solves the static Einstein equations. Section 4 contains a further result on the $n>3$ case, and some discussion regarding potential breakdown of the Horowitz-Myers conjecture. An Appendix contains proofs of results stated in Section 2, including a proof generalizing a result in \cite{AnderssonDahl} that the mass vanishes for any metric that approaches Poincar\'e-Einstein suitably quickly at infinity.

\subsection{Acknowledgements} The author is grateful to Greg Galloway for recommending reference \cite{Corvino}, which proved to be key, and for comments on an earlier draft. He is grateful to the Institut Henri Poincar\'e for hospitality while the majority of this work was completed, and to the Fondation sciences math\'ematiques de Paris for a grant in support of the visit. This work was partially supported by NSERC Discovery Grant RGPIN 203614.

\section{Background}
\setcounter{equation}{0}

\subsection{Asymptotically hyperbolic and Poincar\'e-Einstein metrics}

\noindent A metric $g$ is \emph{conformally compactifiable} is there is a positive function $x:M\to {\mathbb R}$ and a manifold-with-boundary ${\bar M}$ into which $M$ embeds such that (i) ${\bar M}\setminus M = \partial {\bar M}$, (ii) $x$ extends smoothly to ${\partial {\bar M}}$, (iii) $x\vert_{\partial {\bar M}}=0$, (iv) $dx$ is non-vanishing on $\partial {\bar M}$, and (v) ${\bar g}:=x^2 g$ extends continuously to a metric on ${\bar M}$. We refer to $\partial {\bar M}$ as the boundary-at-infinity of $M$. It is sometimes denoted by $\partial_{\infty}M$. The conformal equivalence class of ${\bar g}\vert_{\partial {\bar M}}$ is called the \emph{conformal boundary} of $(M,g)$. We call $x$ a \emph{defining function} for the conformal boundary. We can always arrange that $|dx|^2_{\partial {\bar M}}=1$. If ${\bar g}$ is at least $C^1$, we can solve the differential equation $|dx|^2_{\bar g}=1$ in a collar neighbourhood of ${\partial {\bar M}}$. Then $x$ is called a \emph{special defining function} and $(M,g)$ is called \emph{conformally compactifiable and asymptotically hyperbolic}, or simply \emph{asymptotically hyperbolic}. On a neighbourhood of conformal infinity, the metric can then be written in the form \begin{equation}
\label{eq2.1}
g=x^{-2}\left ( dx^2 \oplus h_x \right )\ .
\end{equation}
Then $dx^2+h_x$ is a metric in Gaussian normal coordinate form, and $g$ is said to be in \emph{normal form} or \emph{Graham-Lee normal form}.

The sectional curvatures of an asymptotically hyperbolic metric approach $-1$ as $x\to 0$. We will be concerned with asymptotically hyperbolic manifold that have Ricci curvature which approaches $-(n-1)g$ sufficiently rapidly near infinity.
\begin{definition}\label{definition2.1}
If
\begin{equation}
\label{eq2.2}
|E|_g\in {\mathcal O}(x^n)\ ,
\end{equation}
where
\begin{equation}
\label{eq2.3}
E:=\ric+n(n-1)g\ ,
\end{equation}
then the manifold is called \emph{asymptotically Poincar\'e-Einstein} or APE \cite{BMW1}. If $E$ vanishes everywhere, the manifold is said to be \emph{Poincar\'e-Einstein}.
\end{definition}

Let $E_{11}=:E(\partial_x,\partial_x)$ and let $E^{\perp}$ denote the projection of $E$ orthogonal to $\partial_x$. Then
\begin{eqnarray}
\label{eq2.4}
E_{11}&=&-\frac12 \tr_{h_x} h_x''+\frac{1}{2x}\tr_{h_x} h_{x}' +\frac14 \vert h_x'\vert^2_{h_x}\ ,\\
\label{eq2.5}
E^{\perp}&=&\ric(h_x)-\frac12 h_x''+\frac{(n-2)}{2x}h_x'+\frac{1}{2x}h_{x} \tr_{h_x}h_x'\\
&&\nonumber +\frac12 h_x'\cdot h_x^{-1}\cdot h_x'-\frac14h_x'\tr_{h_x}h_x'\ .
\end{eqnarray}
Here a prime denotes $\frac{\partial}{\partial x}$ and $h_x'\cdot h_x^{-1}\cdot h_x'$ is the tensor with components $\left ( h_x'\right )_{AC}h^{CD}\left ( h_x'\right )_{DB}$. Then for $A:=\tr_g E = R_g+n(n-1)$ we have
\begin{equation}
\label{eq2.6}
A:=\tr_g E = (n-1)x\tr_hh'-x^2\left [ \tr_h h''-\frac34 \vert h' \vert^2 +\frac14 \left ( \tr_hh'\right )^2-R_h\right ] \ .
\end{equation}

\begin{lemma}\label{lemma2.2}
If the metric (\ref{eq2.1}) is APE and $h_0$ is Einstein, normalized so that $\ric_{h_0}=\lambda (n-2)h_0$ with $\lambda \in \{-1,0,1\}$, then the metric takes the form
\begin{equation}
\label{eq2.7}
g=x^{-2}\left [ dx^2 \oplus \left ( \left ( 1-\frac{\lambda}{4}x^2 \right )^2 h_0 +\frac{x^{n-1}}{(n-1)}\theta+{\mathcal O}(x^n)\right )  \right ]\ ,
\end{equation}
with
\begin{equation}
\label{eq2.8}
\tr_{h_0}\theta=0\ .
\end{equation}
\end{lemma}

\begin{proof}
Using equations (\ref{eq2.4}, \ref{eq2.5}), one can check that whenever $\ric_{h_0}=\lambda (n-2) h_0$ then $E_g=0$ for $g$ given by
\begin{equation}
\label{eq2.9}
g=x^{-2} \left ( dx^2\oplus h_x \right )=x^{-2}\left [ dx^2 \oplus \left ( \left ( 1-\frac{\lambda}{4}x^2 \right )^2 h_0\right ) \right ]\ .
\end{equation}
Thus, $g$ is Einstein. It is well-known that if the Einstein equations are applied order-by-order up to the order required by the APE condition, $h_x$ is uniquely determined by $h_0$ at all orders below $x^{n-1}$, and $tr_{h_0}\theta$ is also uniquely determined (see, e.g., \cite{BMW1}). Therefore, if $g$ as in (\ref{eq2.1}) is an arbitrary APE and $h_0$ is Einstein and normalized as above, all terms in $g$ of order below $x^{n-1}$ must agree with those in (\ref{eq2.9}), and the coefficient of the order $x^{n-1}$ term in $h_x$ must be tracefree (with respect to $h_0$) as dictated by (\ref{eq2.8}).
\end{proof}

In passing, we recall the well-known fact that the Einstein equations, when applied to (\ref{eq2.7}), do not determine the tracefree part of $\theta$, which is free data for the Einstein equations and sometimes called the \emph{Neumann data}. The term \emph{Dirichlet data} refers to $h_0$.

There is no term of the form $x^{n-1}\log x$ in (\ref{eq2.9}), and thus none in (\ref{eq2.7}) either. An arbitrary Poincar\'e-Einstein metric possesses only a polyhomogeneous expansion in $x$ and could have such a log term. Then the coefficient of $x^{n-1}\log x$ term is called the \emph{ambient obstruction tensor}. If this tensor were not to vanish, $g$ would not be \emph{smoothly} conformally compactifiable at $x=0$. However, we have chosen that $h_0$ is Einstein, and then for Poincar\'e-Einstein manifolds the ambient obstruction tensor always vanishes---this is true even if $h_0$ is only conformal to an Einstein metric. Furthermore, in the Poincar\'e-Einstein case if this tensor vanishes then no higher order terms with logarithms appear, so the polyhomogeous expansion of $h_x$ is in fact a Maclaurin series and there is no obstruction to a smooth compactification. Then for APEs, as with Poincar\'e-Einstein metrics, if $g$ as given by (\ref{eq2.1}) is an APE metric and $h_0$ is Einstein or conformally Einstein then $g$ is has vanishing ambient obstruction tensor. However, because APEs do not have to satisfy the Einstein equations at higher order, and arbitrary APE with vanishing obstruction tensor need not have an arbitrarily smooth conformal compactification. Higher order log terms can appear.

The next lemma shows that the scalar curvature of an APE metric falls off slightly faster than $|E|$ at conformal infinity.

\begin{lemma}\label{lemma2.3}
Let $(M,g)$ be APE with $h_{0}$ an Einstein metric obeying $\ric_{h_0}=\lambda (n-2)h_0$ with $\lambda \in \{-1,0,1\}$, so that $g$ is given by (\ref{eq2.7}). Then
\begin{equation}
\label{eq2.10}
A:=R_g+n(n-1)\in {\mathcal O}(x^{n+1})\ ,
\end{equation}
so $A^{(n)}(0)=0$.
\end{lemma}

\begin{proof}
As in (\ref{eq2.1}, \ref{eq2.7}), write $g$ as $g=x^{-2}(dx^2\oplus h_x)$ with
\begin{equation}
\label{eq2.11}
h_x=\left ( 1-\frac{\lambda}{4}x^2 \right )^2 h_0 +\frac{1}{(n-1)}x^{n-1}\theta+\frac{1}{n}x^n\kappa + {\mathcal O}(x^{n+1})\ .
\end{equation}
Plug this into (\ref{eq2.6}). In the resulting expression the terms of order less than $x^{n-1}$ cancel among themselves. Furthermore, the order $x^{n-1}$ terms have coefficient proportional to $\tr_h\theta$, so by equation (\ref{eq2.8}) they vanish. Then the terms of order $x^n$ also simply cancel among themselves.
\end{proof}

Smoothly conformally compactifiable APE metrics such as (\ref{eq2.7}) have the property that they have no terms odd in $x$ of order less than ${\mathcal O}(x^{n-3})$, so $h_x$ has no odd term of order less than ${\mathcal O}(x^{n-1})$. The even terms in (\ref{eq2.7}) below order ${\mathcal O}(x^{n-1})$ are fully determined by the APE condition, but this turns out not to be important for the discussion of conformal invariance later in this section and in the appendix.\footnote
{If $n$ is odd, a general APE metric $g$ can have a term $x^{n-1}\log x$ in the expansion of $h_x$, but the coefficient of this term is determined by $h_0$ and vanishes whenever $g$ is smoothly conformally compactifiable, including whenever $h_0$ is Einstein.}
We therefore define a more general class of metrics (\cite{BMW2}, \cite[cf Definition 2.3.3]{DGH}).

\begin{definition}\label{definition2.4}
A smoothly conformally compactifiable, asymptotically hyperbolic metric in normal form (\ref{eq2.1}) is \emph{partially even} in $x$ if
\begin{equation}
\label{eq2.12}
h_x = h_{(0)}+x^2h_{(2)}+\dots+x^{n-1}h_{(n-1)}+x^n h_{(n)} +{\mathcal O}(x^{n+1})\ .
\end{equation}

\end{definition}

That is, when $g$ is a partially even metric, $h_x$ has no term odd in $x$ of order less than ${\mathcal O}(x^{n-1})$; i.e., $h_x-h_{-x}\in {\mathcal O}(x^{n-1})$. The notion of being partially even is actually independent of the specific choice of special defining function $x$ in the definition (\cite[Proposition 2.3.4]{DGH} and Proposition \ref{proposition2.6}). Every APE metric is partially even.

\subsection{The Wang mass-energy}

\noindent The Wang mass was defined for APE metrics with round sphere conformal infinity and zero Neumann data in the original paper \cite{Wang}. However, it generalizes to other conformal boundaries and to more general APEs. The case of an APE metric with flat toroidal conformal infinity and zero Neumann data was studied in \cite{BW}.

\begin{definition}\label{definition2.5}
For metrics of the form (\ref{eq2.1}, \ref{eq2.11}) with $h_0$ a closed Einstein $(n-1)$-metric normalized so that $\ric(h_0)=(n-2)\lambda h_0$, the \emph{Wang mass-energy} (or simply \emph{mass}) is
\begin{equation}
\label{eq2.13}
m\equiv m[g]:=\int_{\partial_{\infty}M}\tr_{h_0}\kappa dV(h_0) \ ,
\end{equation}
where $dV(h_0)$ is the volume element for the metric $h_0$. The quantity $\mu:=\tr_{h_0}\kappa$ is called the \emph{mass aspect function}.
\end{definition}

Wang's original definition was given in the special case of $\lambda=1$, $n=3$, and $\partial_{\infty}M\equiv S^2$. Near infinity, one recovers the isometry group of hyperbolic 3-space, $\SO(3,1)$, or $\SO(n,1)$ in arbitrary dimensions, in the approximation of small defining function $x$. Now in the related context of asymptotically anti-de Sitter $(n+1)$-dimensional spacetimes, there is a conserved charge for each element of the vector space of asymptotic Killing fields, i.e., the Lie algebra $\so(n,2)$ \cite{AshtekarMagnon}. Similarly, in the current setting, one can define a mass integral associated to each element of a vector space \cite{CH} (of dimension $n+1$ in this case, it turns out). These integrals can be combined as an $(n+1)$-vector whose signed $\so(n,1)$-invariant norm yields the invariant that Wang called mass. The sign is given by the mass integral which is the ``timelike'' component of the vector.\footnote
{Explicitly, setting $\lambda=1$ and $n=3$, the substitution $\sinh r=\frac{x}{1-x^2/4}$ yields Wang's definition \cite[pp 275--276]{Wang} when $\kappa$ is constant over conformal infinity. Wang differs from most work in the field, including the present paper, by using the term \emph{special defining function} to describe this coordinate $r$ rather than the Gaussian normal coordinate $x$.}
It is this timelike component that appears in the definition above and which we call \emph{mass-energy} or simply \emph{mass}. The other mass integrals measure the ``dipole moment'' of the mass aspect $\kappa$. They vanish, for example, when $\kappa$ is constant on conformal infinity.

Such considerations are not needed when $\lambda=0$, which will be our main interest here, nor are they needed when $\lambda=-1$. In these cases the vector space is one-dimensional and the mass integral above is the only one up to scaling \cite[Section 3]{CH}, hence we simply use the term \emph{mass} for this integral. Also, \cite{ACG} used the term \emph{mass} for this integral in their study of the $\lambda=1$ case under the assumption that $\kappa$ had constant sign.

In dimension $n=3$ with $\lambda=1$, Wang \cite[equation (33)]{Wang} noticed that the mass aspect has a nice weighted invariance property with respect to conformal changes in the metric on the boundary at infinity. The following result generalizes this observation in a number of ways.

\begin{proposition}\label{proposition2.6}
Let $n=\dim M$ be odd and let $g$ be partially even in the special defining function $x$ inducing the metric $h_0$ on $\partial_{\infty} M$, so that $g=\frac{1}{x^2}\left ( dx^2\oplus h_x\right )$ and
\begin{equation}
\label{eq2.14}
\begin{split}
h_x =&\, h_{(0)}+x^2h_{(2)}+\dots+x^{n-1}h_{(n-1)}+x^n h_{(n)} +{\mathcal O}(x^{n+1})\\
=&\, h_{(0)} +x^2h_{(2)}+\dots+\frac{1}{(n-1)}\theta x^{n-1} +\frac{1}{n}\kappa x^n +{\mathcal O}(x^{n+1}) \ .
\end{split}
\end{equation}
Let ${\hat x}$ be another special defining function, inducing metric ${\hat h}_{\hat x}$ such that ${\hat h}_0:=e^{2\omega_0}h_0$; i.e., ${\hat h}_0$ is conformal to $h_0$ on $\partial_{\infty} M$. Then ${\hat h}_{\hat x}$ is partially even in ${\hat x}$ and $\mu:=\tr_{h_{(0)}}\kappa$ is conformally invariant with weight $-n$; i.e., ${\hat \mu}:=\tr_{{\hat h}_{(0)}}{\hat \kappa}=e^{-n\omega_0}\mu$.
\end{proposition}

This is a deviation from the main direction of our discussion, since Conjecture \ref{conjecture1.1} breaks conformal invariance down to homothetic invariance by requiring a flat boundary. We therefore relegate the proof to an appendix. The proposition does not elevate $\mu$ to the status of a weighted conformal invariant of the boundary since nothing in the assumptions determines $\mu$ in terms of the boundary conformal class $[h_{(0)}]$ alone. The statement is merely that \emph{if conformal variations are the only variations allowed}, then the resulting change in $\mu$ is simple when $g$ is partially even and $n$ is odd.

Now say that we attempt to promote $\mu$ to a weighted conformal invariant of the boundary by imposing the Einstein equations (at least to sufficient order in powers of $x$) in order to fix $\kappa$ (or at least its trace). This strategy would work, provided that $\kappa$ is only dependent on the conformal class of $h_0$ and not dependent on the Neumann data (the tracefree part of $\theta^{\rm TF}$). The next proposition shows that this strategy does produce a conformal invariant at least when the conformal class of $h_0$ admits an Einstein metric, but in that case the invariant is always zero. This is well-known when conformal infinity is a spherical space form \cite{AnderssonDahl} and $(M,g)$ admits the appropriate spinor structure for the Witten technique \cite{Witten1}. The proposition holds in either even or odd dimension $n$, but if $n$ is even then the special defining function $x$ used in the expansion which defines the mass should be the one corresponding to the Einstein metric in $[h_0]$.

\begin{proposition}\label{proposition2.7}
Let $(M,g)$ be Poincar\'e-Einstein and let the induced conformal class on $\partial_{\infty}M$ admit an Einstein metric $\ric(h_0)=(n-2)\lambda h_0$. Then the mass (as computed with respect to $h_0$ if $n$ is even) of $(M,g)$ is zero.
\end{proposition}

\begin{remark}\label{remark2.8}
The Poincar\'e-Einstein condition can be relaxed. All that is needed for Proposition \ref{proposition2.7} is that $(M,g)$ be Poincar\'e-Einstein to one more order in the expansion of $h_x$ than generic APE metrics; i.e., $|E|:=\vert \ric + (n-1)g\vert \in {\mathcal O}(x^{n+1})$.
\end{remark}

Again, the proof is to be found in the appendix.

In what follows, we obviously do not want to impose the Einstein equations beyond order $x^n$, since then the mass would be zero. Instead, we impose them only to order $x^n$, so that $g$ is APE. We will as well impose a global condition on scalar curvature and attempt to show that, under these conditions, a complete $k=0$ mass minimizing metric exists. To show rigidity, we must show that it obeys certain definite equations to all orders in $x$, but these will not be the Einstein equations, and consequently the minimal mass will not be zero (indeed, it will be negative!). These equations are the subject of the next subsection.

\subsection{Static metrics}

\noindent We recall briefly that an asymptotically hyperbolic metric $g$ on an $n$-manifold is said to be a \emph{static Einstein metric} if there is a positive function $N$, called the \emph{lapse}, such that the pair $(g,N)$ obeys the system
\begin{eqnarray}
\label{eq2.15} N{\ric}&=& \hess N -nNg\ ,\\
\label{eq2.16} \Delta N &=& nN\ .
\end{eqnarray}
Here the Laplacian $\Delta:= \tr_g \hess$ where $\hess$ is the Hessian, and $\ric=\ric_g$ is the Ricci tensor of $g$. Asymptotically hyperbolic static Einstein metrics have constant negative scalar curvature. We have rescaled $g$ here so that its scalar curvature obeys $R_g+n(n-1)=0$. These equations are equivalent to the equation
\begin{equation}
\label{eq2.17}
\hess N -g\Delta N-N\ric=0\ .
\end{equation}
As well as being called lapse functions, solutions $N$ of (\ref{eq2.17}) are sometimes called \emph{static potentials}.

These equations are singular at $N=0$. On domains that do not contain singular points, one can take $N=:e^u$. Then the static Einstein equations take the form
\begin{eqnarray}
\label{eq2.18} 0&=&\ric - du\otimes du -\hess u +ng\ ,\\
\label{eq2.19} 0&=&\Delta u +|du|^2-n\ .
\end{eqnarray}

Solutions of these equations describe $(n+1)$-dimensional spacetimes with metric $-N^2dt^2\oplus g$, but can also describe $(n+1)$-dimensional Riemannian manifolds $N^2dt^2\oplus g$ where $\frac{\partial}{\partial t}$ is a hypersurface-orthogonal Killing vector field, possibly with fixed points where $N=0$; these are always totally geodesic hypersurfaces \cite[Lemma 2.1.(i)]{GM} and are called \emph{bolts}. If the fixed point set is non-empty, the metric is called \emph{locally static}, otherwise it is \emph{globally static}. Smoothness at a bolt requires the domain of $t$ to be a circle of a certain definite circumference.

\begin{proposition}\label{proposition2.9}
Let $(M,g,N)$ solve the static Einstein equations, where $(M,g)$ is an APE $n$-manifold such that $g=\frac{1}{x^2}(dx^2\oplus h_x)$, where $\ric_{h_0}=(n-2)\lambda h_0$, $\lambda \in \{ -1,0,1\}$, and $xN\to 1$ as $x\to 0$. Then $x^2g$ extends smoothly to $x=0$ and
\begin{equation}
\label{eq2.20}
\begin{split}
h_x=&\, \left [ \left ( 1-\frac14 \lambda x^2\right )^2 h_0 +\frac{1}{n}\kappa x^n\right ] \oplus x^2N^2d\tau^2 +{\mathcal O}(x^{n+1})\ ,\\
N=&\, \frac{1}{x}\left ( 1+\frac14 \lambda x^2 \right )-\frac{\mu}{2n} x^{n-1} +{\mathcal O} (x^n)\ , \\
\mu =&\, \tr_{h_0}\kappa\ .
\end{split}
\end{equation}
\end{proposition}

The proof is given in the appendix. The $\lambda=0$ version of this result appeared in \cite{GW}. We do not need that $(M,g,N)$ is a smooth global solution of the static Einstein equations. All that is needed is that $N^2d\tau^2 \oplus g$ obeys the Einstein equations to order $x^n$ inclusive on a neighbourhood of conformal infinity---this is the APE condition on an $(n+1)$-manifold---and that $\frac{\partial}{\partial \tau}$ is a Killing field (to sufficient order) on this neighbourhood. The importance of this result is that the fall-off conditions imply that the mass aspect $\mu=\tr_{h_0}\kappa$ of $g$ appears as it does above in the expression for $N$. In the AdS/CFT correspondence, the coefficients of the order $x^n$ terms in the Fefferman-Graham expansion \cite{FG1, FG2} of an $(n+1)$-dimensional Poincar\'e-Einstein metric combine to give the vacuum expectation value of the boundary CFT stress-energy tensor. The coefficient of the order $x^n$ term in the lapse function is then the \emph{vacuum expectation value of the holographic energy density}.

\begin{remark}\label{remark2.10}
Under the conditions of Proposition \ref{proposition2.9}, Wang's mass aspect equals the vacuum expectation value of the holographic energy density.
\end{remark}

Our eventual interest will be in complete manifolds with a single asymptotic end on which $N$ grows like $1/x$, hence the condition above (indeed, $1/N$ is often used as a defining function, though not a special defining function in our sense; see, e.g., \cite{CS, GSW1, GSW2, GSW3}). We will also take special interest in the case where the fixed point set of $\frac{\partial}{\partial \tau}$ is empty, so that $N$ has no zeroes. Since $N$ grows at infinity, it therefore will be bounded away from zero on $M$ and $u$ will be bounded below.

\subsection{AdS Solitons}

For $\dim M=n\ge 3$, consider the family of metrics on $M$ given by
\begin{equation}
 \label{eq2.21}
ds^2=\frac{dr^2}{r^2\left ( 1-\frac{1}{r^n}\right )}
+r^2\left [ \left ( 1-\frac{1}{r^n}\right ) d\xi^2
+\sum_{i=3}^{n} d\phi_i^2\right ] \ .
\end{equation}
The domains of the coordinates are $r\in [1,\infty)$, $\xi\in [0,4\pi/n]$, and $\phi_i\in [0,a_i]$ where $0<a_3\le\dots\le a_n$. We identify $\phi_i\sim\phi_i+a_i$. The locus $r=1$ is the \emph{central torus} or, if $n=3$, the \emph{central circle}. The domain of $\xi$ is chosen so that the metric is smooth on the central torus. Then (\ref{eq2.21}) represents a family (parametrized by the $a_i$) of smooth metrics on ${\mathbb R}^2\times T^{n-2}$, where $T^{n-2}$ is an $(n-2)$-torus. A parameter, sometimes denoted $r_0$ or $M$, often appears in descriptions of the metric (\ref{eq2.21}) \cite{HM}, but has no significance and can be removed by rescaling the coordinates. Another parameter, $\ell$, the {\it radius of curvature at infinity}, also sometimes appears but can be removed by homothetic rescaling of the metric. We will rescale all asymptotically hyperbolic metrics so that sectional curvatures approach $-1$ on asymptotic ends.

Another parametrization of the metrics (\ref{eq2.21}) is in terms of distance $\rho$ from the central torus. Then we have
\begin{equation}
\label{eq2.22}
ds^2 = d\rho^2+\left (\cosh\frac{n\rho}{2}\right )^{4/n} \left [ \left ( \tanh \frac{n\rho}{2}\right )^2 d\xi^2 +\sum_{i=3}^n d\phi_i^2 \right ] \ ,
\end{equation}
where $\rho\in [0,\infty)$.

We will refer to the above metrics as \emph{Horowitz-Myers geons}. They are asymptotically (locally) hyperbolic and have scalar curvature $R=-n(n-1)$. It is an easy exercise to find a special defining function for a Horowitz-Myers geon. The result is\footnote
{In many papers, the lapse is used as a defining function for conformal infinity \cite{CS}. Here that would amount to the choice $x=r$, which is subtlely different from special defining function (i.e., Gaussian normal coordinate) choice.}
\begin{equation}
\label{eq2.23}
x:=4^{1/n}\left [r^{n/2}-\sqrt{r^n-1} \right ]^{2/n}=4^{1/n}r\left [ 1-\sqrt{1-1/r^n}\right ]^{2/n}=4^{1/n}e^{-\rho}\ .
\end{equation}
Using this as a coordinate, the geon metric takes the form
\begin{equation}
\label{eq2.24}
\begin{split}
ds^2=&\, x^{-2} \left \{ dx^2+\left ( 1+\frac{x^n}{4} \right )^{4/n} \left [ \left ( \frac{1-x^n/4}{1+x^n/4} \right )^2 d\xi^2+\sum\limits_{i=3}^nd\phi_i^2 \right ] \right \}\ , \\
x\in&\, [0,4^{1/n}]\ , \\
\xi\in&\, [0,4\pi/n]\ , \\
\phi_k \in&\, [0,\alpha_k]\ , \ k\in \{3,\dots,n\}\ .
\end{split}
\end{equation}
By expanding in powers of $x$ and comparing to (\ref{eq2.11}), we can read off that
\begin{equation}
\label{eq2.25}
\begin{split}
\theta=&\, 0\\
\kappa=&\, -(n-1)d\xi^2+\sum\limits_{i=3}^nd\phi_i^2\ , \\
\tr_{\delta} \kappa =&\, -(n-1)+(n-2)=-1\ .
\end{split}
\end{equation}
Then from (\ref{eq2.13}) we get that the Wang mass of a Horowitz-Myers geon is
\begin{equation}
\label{eq2.26}
m=-\frac{4\pi}{n}\sum_{i=3}^n a_i <0\ .
\end{equation}

Finally, we can obtain a Poincar\'e-Einstein metric by appending an extra factor of ${\mathbb R}$ to this manifold, say with coordinate $\tau$, and appending a corresponding factor to the metric (\ref{eq2.24}). The result is
\begin{equation}
\label{eq2.27}
g=\pm N^2 d\tau^2+ds^2\ ,
\end{equation}
where $ds^2$ is the metric (\ref{eq2.21}), or equivalently (\ref{eq2.24}), and the lapse function $N$ is
\begin{equation}
\label{eq2.28}
N=r=\frac{1}{x}\left ( 1+\frac{x^n}{4}\right )^{2/n}\ .
\end{equation}
The Lorentzian version, obtained by choosing the negative sign in (\ref{eq2.27}), is called an \emph{AdS soliton}. One can read off from (\ref{eq2.28}) that $N=\frac{1}{x}\left (1+\frac{1}{2n}x^n+{\mathcal O}(x^{2n}) \right )$, which confirms that the lapse $N$ agrees with Proposition \ref{proposition2.9} since $\mu=-1$.

The Lorentzian-signature metric $-r^2d\tau^2\oplus g$ is called an {\it AdS soliton} \cite{HM}. It is negative Einstein and asymptotically locally anti-de Sitter with boundary-at-infinity given by a flat $(n-2)$-torus (producted with the time direction). AdS solitons admit $\frac{\partial}{\partial \tau}$ as a hypersurface-orthogonal timelike Killing vector field with complete orbits. Hence AdS solitons are globally static. The time-symmetric slices are Horowitz-Myers geons. The lapse function is simply $N=r$.

\subsection{Standard approaches}

Now that we have an explicit description of Horowitz-Myers geons, it becomes easier to see why standard approaches to rigidity theorems will not work, and it allows us to see what new ingredient is needed. We will mostly limit ourselves to discussing rigidity rather than the corresponding positive mass theorems, though a motivating principle for this paper is that the two can be linked, as they so clearly are in the Witten spinor approach.

The negativity of the mass of Horowitz-Myers geons is remarkable. Spinorial techniques are commonly employed both to prove positive mass theorems and to prove the corresponding rigidity of the mass mimimizer, going all the way back to \cite{Witten1}. However, while Horowitz-Myers geons admit spinor structure, they do not admit the right kind. Globally defined spinors on a Horowitz-Myers geon must be anti-periodic under transport along a cycle at infinity tangent to $\frac{\partial}{\partial \xi}$, but the Witten approach to proving positivity of mass, and rigidity, requires there to exist global solutions of the Witten equation which are constant under transport at infinity. Since the torus at infinity is flat, constant spinors do not acquire a phase and cannot be anti-periodic under transport along this cycle.

Now consider the approach used in \cite{ACG} to prove rigidity of the zero-mass case once the positive energy theorem for asymptotically hyperbolic manifolds with spherical conformal infinity was proved. They devised a perturbation of the metric which increased the scalar curvature of any non-Einstein, constant (negative) scalar curvature metric while keeping the mass zero. Then a conformal transformation can be used to return the scalar curvature to $-n(n-1)$ while lowering the mass. This makes an initially zero mass become negative, contradicting the positive energy theorem for asymptotically hyperbolic metrics with round sphere conformal infinity \cite{Wang, CH}. The perturbation constructed in \cite{ACG} was essentially an infinitesimal version of a normalized Ricci flow (this flow was studied by \cite{Bahuaud}). But the flow of \cite{Bahuaud} will change the mass if the mass is not zero before the perturbation \cite{BW}, and will cause a negative mass to increase (to become closer to zero), so this process does not necessarily produce a net decrease in mass. That is why a na\"ive application of the \cite{ACG} approach will not work.

But in the next section we will consider a different perturbation, one which leaves the mass constant and strictly increases when applied to a metric which obeys $R+n(n-1)=0$ but does not obey the static Einstein equations, and vanishes when that metric does obey the static Einstein equations.

\section{Variation of the metric and the scalar curvature}
\setcounter{equation}{0}

\subsection{Corvino's variation for constant scalar curvature metrics}

\noindent Following a well-worn path, the idea of the rigidity proof is to show that if the mass equals the supposed minimum value $m=m_0$ but the scalar curvature obeys $A:=R+n(n-1)\ge 0$ with strict inequality at least somewhere, then there is a conformal variation of the metric which lowers the mass and preserves the inequality $A:=R+n(n-1)\ge 0$ (indeed, the technique yields $A=0$ everywhere). The variation is constructed by solving the Yamabe equation for these asymptotics. This would violate the assertion that $m_0$ is the least mass.

Since the asymptotically hyperbolic structure of the metrics under consideration requires that the sectional curvatures approach $-1$ at conformal infinity, the scalar curvature is constant iff it equals $-n(n-1)$ everywhere. Whenever the scalar curvature is nonconstant, we can proceed directly to the next step, which is to solve the Yamabe problem.

In this subsection, we show that if the scalar curvature is constant but the metric does not satisfy the static Einstein equations, then there is a variation of the metric which produces scalar curvature obeying $A=R+n(n-1)\ge 0$ with strict inequality somewhere.

\begin{lemma}\label{lemma3.1}
Let $(M,g_0)$ be APE, with $g_0\in C^{\infty}(M)$, $A:=R+n(n-1)\equiv 0$, and Wang mass $m(g_0)=m_0$. Furthermore, say that there does not exist any nonconstant function $N$ such that $(N,g_0)$ satisfies equations (\ref{eq2.15}, \ref{eq2.16}) (equivalently, \ref{eq2.17}). Then for any sufficiently small $\epsilon>0$ and any function $0<{\mathcal A}<\epsilon$ on $\Omega$, there is a metric $g=g_0+h$ such that ${\mathcal A}-n(n-1)$ is the scalar curvature of $g$ on $\Omega$. Furthermore $g\in C^{\infty}(M)$, $(M,g)$ is APE, and the Wang mass of $g$ is $m(g)=m_0$.
\end{lemma}

\begin{proof} The proof is due to Corvino \cite[Theorem 4, and the Remark that follows it]{Corvino}. Consider the operator $L^*_{g_0}:H_{\rm loc}^2(\Omega)\to {\mathcal L}_{\rm loc}^2(\Omega)$ (see \cite[Section 2.1]{Corvino} for definitions of the function spaces ), for $\Omega$ some compact domain in $M$. Here $g_0$ is any constant scalar curvature metric, so we take $R_{g_0}=-n(n-1)$. The operator is given by the expression
\begin{equation}
\label{eq3.1}
L^*_{g_0}(N):=\hess_{g_0} N -g_0\Delta_{g_0} N -N\ric_{g_0}\ .
\end{equation}
If there is no nonconstant $N$ such that $(N,g_0)$ satisfies (\ref{eq2.17}) in $\Omega$, then the the kernel of this operator is trivial. This operator is the formal $L^2$-adjoint of the linearized scalar curvature map about $g_0$. Corvino shows that, when the kernel is trivial, then for any sufficiently small $\epsilon>0$ and any function ${\mathcal R}$ with ${\mathcal R}$ $\epsilon$-close to $R_{g_0}$ in a suitable norm on $\Omega$, there is a metric $g$ such that ${\mathcal R}$ is the scalar curvature of $g$ on $\Omega$. For our purposes, $R_{g_0}=-n(n-1)$ and we set ${\mathcal A}:={\mathcal R}+n(n-1)$. Moreover, the perturbation $g-g_0=h$ vanishes outside $\Omega$, so $g$ is APE with Wang mass $m_0$
\end{proof}

\subsection{Yamabe variation of non-constant scalar curvature metrics}

\noindent We may now assume that either $g=g_0$ where $L_{g_0}^*$ has nontrivial kernel or we have an APE metric $g$ with nonconstant scalar curvature such that $A:=R+n(n-1)\ge 0$. In the latter case, we now construct a new metric with constant scalar curvature with mass less than that of the original metric. We note that, taking the trace of (\ref{eq2.18}) and using (\ref{eq2.19}), we see that any solution of the static Einstein equations must have constant scalar curvature, so $L_{g}^*$ cannot have nontrivial kernel.

\begin{lemma}\label{lemma3.2}
Let $(M,g)$ be an APE with toroidal infinity, with nonconstant scalar curvature $R_g\ge -n(n-1)$. Then there is an APE metric ${\hat g}$ which is conformal to $g$ and which obeys ${\hat A}:=A_{\hat g}:=R_{\hat g}+n(n-1)\equiv 0$.
\end{lemma}

\begin{proof}
Apply \cite[Theorem A]{AM} to the metric $(M,g)$ from Lemma \ref{lemma3.1}. This shows that there is a metric ${\hat g}$ conformal to $g$ with ${\hat A}=0$.

We next confirm that the metric is APE. Say the conformal factor is ${\hat g}=\varphi^{\frac{4}{(n-2)}}g$. Then $\varphi$ solves the Yamabe equation
\begin{equation}
\label{eq3.2}
Y(\varphi):=-4\left ( \frac{n-1}{n-2}\right )\Delta \varphi + n(n-1)\left ( \varphi^{\frac{(n+2)}{(n-2)}}-\varphi\right ) +A(g)\varphi=0 \ .
\end{equation}
Using equation (2.1), we may expand the Laplacian as
\begin{equation}
\label{eq3.3}
\Delta \varphi = x^2 \varphi'' +\frac{x^2}{\sqrt{h_x}}\frac{\partial \sqrt{h_x}}{\partial x} \varphi'-(n-2)x\varphi'+x^2 \Delta_{h_x}\varphi\ ,
\end{equation}
where $\varphi':=\frac{\partial \varphi}{\partial x}$ and $\Delta_{h_x}$ is the scalar Laplacian defined by the metric $h_x$. We seek a solution with $\varphi(0)=1$. Let $v(x):=\varphi-1$ so that we seek to solve
\begin{equation}
\label{eq3.4}
\begin{split}
0=&\, x^2 v'' +\frac{x^2}{\sqrt{h_x}}\frac{\partial \sqrt{h_x}}{\partial x} v'-(n-2)x\varphi'+x^2 \Delta_{h_x}v\\
&\, -\frac{n(n-2)}{4}(1+v)\left [\left ( (1+v)^{\frac{4}{(n-2)}}-1\right ) +  \frac{A(g)}{n(n-1)} \right ]\ .
\end{split}
\end{equation}

Equation (\ref{eq3.4}) is a singular PDE at conformal infinity $x=0$. We now take $l$ derivatives of (\ref{eq3.4}) using the expansion (\ref{eq3.3}), and evaluate the result at $x=0$, with $\varphi(0)=1$. A calculation yields
\begin{equation}
\label{eq3.5}
\begin{split}
0=&\, -4 \left ( l^2-(n-1)l-n\right )v^{(l)}(0)+\left ( \frac{n-2}{n-1} \right ) A^{(l)}\\
&\,+ F\left ( v'(0), v''(0),\dots,v^{(l-1)}(0)\right ) \ ,\ l=1,2,\dots\ ,
\end{split}
\end{equation}
where $v^{(k)}(0)$ denotes $\frac{\partial^k}{\partial x^k} \big \vert_{x=0} v$. The function $F$ depends on the first $l-1$ $x$-derivatives of $v$ at $x=0$ (and on $A'(0),\dots,A^{(l-1)}(0)$, recalling that $A(0)=0$ for an asymptotically hyperbolic metric). Moreover, $F$ is homogeneous; i.e., $F(0,\dots,0)=0$; and furthermore since $g$ is APE with toroidal infinity, we have $A^{(l)}(0)=0$ if $l\le n$ by Lemma \ref{lemma2.3}. Thus we obtain $v^{(l)}(0)=0$ for $l< n$. For $l=n$, the indicial equation $l^2-(n-1)l-n=0$ has a root, so $v^{(n)}$ is not determined. However, we may now conclude that the expansions for $g$ and ${\hat g}$ in their respective special defining functions cannot differ below order $x^n$, so ${\hat g}$ is APE. \end{proof}

Then we have the following result, which shows that under the conditions laid out in Lemma \ref{lemma3.1}, there is a variation which lowers the mass and preserves the boundedness condition on scalar curvature.

\begin{lemma}\label{lemma3.3}
Let $g$ and ${\hat g}=\varphi^{\frac{4}{(n-2)}}g$ be as in Lemma \ref{lemma3.2}. Then the Wang mass ${\hat m}$ of ${\hat g}$ is less than the Wang mass $m$ of $g$; i.e., ${\hat m}<m$.
\end{lemma}

\begin{proof}
The argument is essentially that given in \cite[Lemmata 3.11--3.13]{ACG}. First note that by the strong maximum principle applied to (\ref{eq3.2}), we have $\varphi<1$ on $M$.

Next, define a one-parameter family of functions
\begin{equation}
\label{eq3.6}
w_{(\alpha)}:=1-\alpha \left ( x^n+nx^{n+1}\right )\ .
\end{equation}
Then by explicit calculation we have
\begin{equation}
\label{eq3.7}
\Delta w_{(\alpha)}=-\alpha \left (nx^n+2n(n+1)x^{n+1}+{\mathcal O}(x^{n+2})\right )\ .
\end{equation}
Hence
\begin{equation}
\label{eq3.8}
\begin{split}
Y(w_{(\alpha)})=&\, 4\alpha n\left ( \frac{n-1}{n-2}\right )x^n+8\alpha n\frac{(n+1)(n-1)}{(n-2)}x^{n+1}-4\alpha n\left ( \frac{n-1}{n-2}\right )x^n\\
&\, -4\alpha n^2 \left ( \frac{n-1}{n-2}\right )x^{n+1}+Aw_{(\alpha)}+\alpha {\mathcal O}(x^{n+2})\\
=&\, 4\alpha n\frac{(n-1)(n+2)}{(n-2)}x^{n+1}+Aw_{(\alpha)}+\alpha {\mathcal O}(x^{n+2})\\
>&\, 0
\end{split}
\end{equation}
for small $x>0$ and $\alpha\in (0,1]$ (using that $A\ge 0$). Thus, for any $\alpha\in (0,1]$, $w_{(\alpha)}$ is a subsolution for $Y(\varphi)=0$, $\varphi(0)=1$.

Now choose a compact level set $x=\epsilon$ near infinity, where $\epsilon>0$ is small. Denote the level set by ${\mathcal S}$. Since $\varphi<1$ on $M$, then $\varphi\le 1-c^2 <0$ on ${\mathcal S}$, for some constant $c>0$. Choose a fixed $\alpha=\alpha_1>0$ such that $1-c^2\le w_{(\alpha_1)}<1$ on ${\mathcal S}$. Now consider the function $f(x,y^A):=\varphi(x,y^A)-w_{(\alpha_1)}(x)$. If it has a positive maximum for $x\in [0,\epsilon]$, then this maximum must occur for some $x\in (0,\epsilon)$, since $f(0,y^A)=0$ and $f(\epsilon,y^A)\le 0$. But by taking the difference of equations (\ref{eq3.2}) and (\ref{eq3.8}) we have for $x\in (0,\epsilon)$ that
\begin{equation}
\label{eq3.9}
\begin{split}
&\,-4 \left ( \frac{n-1}{n-2}\right ) \Delta f + n(n-1) \left [ \left ( \varphi^{\frac{n+2}{n-2}}-\varphi \right ) - \left ( w_{(\alpha_1)}^{\frac{n+2}{n-2}}-w_{(\alpha_1)}\right ) \right ] +Af\\
&\, \le -4\alpha_1 n \frac{(n-1)(n+2)}{(n-2)}x^{n+1}+\alpha_1 {\mathcal O}(x^{n+2})\\
&\, <0\ .
\end{split}
\end{equation}
A routine calculation, using that $\frac{\varphi-w_{(\alpha_1)}}{\varphi}\in {\mathcal O}(x^n)$, simplifies the term in square brackets as follows.
\begin{equation}
\label{eq3.10}
\begin{split}
&\, \left ( \varphi^{\frac{n+2}{n-2}}-\varphi \right ) - \left ( w_{(\alpha_1)}^{\frac{n+2}{n-2}}-w_{(\alpha_1)}\right ) \\
=&\, \varphi^{\frac{n+2}{n-2}}\left [ 1 - \left ( 1- \frac{(\varphi-w_{(\alpha_1)})}{\varphi}\right )^{\frac{n+2}{n-2}} \right ]-\left ( \varphi - w_{(\alpha_1)} \right )\\
=&\, \varphi^{\frac{n+2}{n-2}}\left [ 1 - \left ( 1- \left ( \frac{n+2}{n-2}\right ) \frac{(\varphi-w_{(\alpha_1)})}{\varphi}\right ) \right ] -\left ( \varphi - w_{(\alpha_1)} \right )+{\mathcal O}(x^{2n})\\
=&\, \left [ \left ( \frac{n+2}{n-2}\right ) \varphi^{\frac{4}{(n-2)}}-1\right ] \left ( \varphi - w_{(\alpha_1)} \right ) +{\mathcal O}(x^{2n})\\
=&\, \left [ \left ( \frac{n+2}{n-2}\right ) \varphi^{\frac{4}{(n-2)}}-1\right ] f +{\mathcal O}(x^{2n})\ .
\end{split}
\end{equation}
Thus, inequality (\ref{eq3.9}) becomes
\begin{equation}
\label{eq3.11}
-4 \left ( \frac{n-1}{n-2}\right ) \Delta f +\left [ n(n-1) \left ( \left ( \frac{n+2}{n-2}\right ) \varphi^{\frac{4}{(n-2)}}-1\right ) +A\right ] f<0\ .
\end{equation}
By the maximum principle, $f$ can have a positive interior maximum only if
\begin{equation}
\label{eq3.12}
n(n-1) \left ( \left ( \frac{n+2}{n-2}\right ) \varphi^{\frac{4}{(n-2)}}-1\right ) +A < 0
\end{equation}
at the maximum. But since $A\ge 0$ and $\varphi=1+{\mathcal O}(x^n)$ with $x\in (0,\epsilon)$ with $\epsilon$ small, this coefficient is in fact positive for all $x\in (0,\epsilon)$. Thus $f\le 0$ for all $x\in (0,\epsilon)$. Hence
\begin{equation}
\label{eq3.13}
\varphi\le w_{(\alpha_1)}=1-\alpha_1\left ( x^n+nx^{n-1}\right )
\end{equation}
or all $x\in (0,\epsilon)$. This implies that
\begin{equation}
\label{eq3.14}
\varphi=1-\frac{1}{n}a_{(n)}x^n+{\mathcal O}(x^{n+1})\ , \ a_{(n)}\ge \alpha_1>0\ .
\end{equation}

Finally, we now have
\begin{equation}
\label{eq3.15}
\begin{split}
{\hat g}=&\, \varphi^{\frac{4}{(n-2)}}g\\
=&\, \left ( 1-\frac{4}{n(n-2)}a_{(n)}x^n\right )\frac{1}{x^2}\left [ dx^2\oplus \left (\delta +\frac{1}{(n-1)}x^{n-1}\theta+\frac{1}{n}\kappa x^n\ +{\mathcal O}(x^{n+1})\right ) \right ]\ .
\end{split}
\end{equation}
We have to recompute the special defining function. It suffices to do this to order $x^n$ inclusive. That is, we write
\begin{equation}
\label{eq3.16}
\begin{split}
\frac{d{\hat x}}{{\hat x}}=&\, \left ( 1-\frac{4}{n(n-2)}a_{(n)}x^n\right )^{1/2}\frac{dx}{x}\\
=&\, \left ( 1-\frac{2}{n(n-2)}a_{(n)}x^n+{\mathcal O}(x^{n+1})\right )\frac{dx}{x}\ ,
\end{split}
\end{equation}
which integrates to yield
\begin{equation}
\label{eq3.17}
{\hat x}=x-\frac{2a_{(n)}}{n^2(n-2)}x^{n+1}+{\mathcal O}(x^{n+2})\ .
\end{equation}
This in turn gives
\begin{equation}
\label{eq3.18}
x={\hat x}+\frac{2a_{(n)}}{n^2(n-2)}{\hat x}^{n+1}+{\mathcal O}({\hat x}^{n+2})\ .
\end{equation}
Using (\ref{eq3.18}) and the first line of (\ref{eq3.16}) in (\ref{eq3.15}), we have
\begin{equation}
\begin{split}
\label{eq3.19}
{\hat g}=&\, \frac{d{\hat x}^2}{{\hat x}^2}\oplus \left [ \left ( 1-\frac{8a_{(n)}}{n(n-2)}x^n \right ) \delta +\frac{1}{(n-1)}{\hat x}^{n-1}\theta +\frac{1}{n}\kappa {\hat x}^n\ +{\mathcal O}({\hat x}^{n+1}) \right ]\\
=&\, \frac{d{\hat x}^2}{{\hat x}^2}\oplus \left [ \delta +\frac{1}{(n-1)}{\hat x}^{n-1}\theta +\frac{1}{n}{\hat \kappa} {\hat x}^n\ +{\mathcal O}({\hat x}^{n+1}) \right ]\\
{\hat \kappa}:=&\, \kappa - \frac{8a_{(n)}}{n(n-2)} \delta\ .
\end{split}
\end{equation}
Then $\tr_{\delta}{\hat \kappa}=\tr_{\delta}\kappa-\frac{8(n-1)}{n(n-2)}a_n<\tr_{\delta}\kappa$ and by (\ref{eq2.13}) we have ${\hat m}<m$.
\end{proof}

\subsection{Proof of Theorem \ref{theorem1.2}}

\begin{proof}
Say that $g_0$ is APE with mass $m_0$, with $A_{g_0}\ge 0$, and such that the kernel of $L_g^*$ is trivial. If $A_{g_0}$ is somewhere nonzero, set $g:=g_0$; otherwise invoke Lemma \ref{lemma3.1} to produce a metric $g$ with $A_g$ somewhere nonzero and everywhere nonnegative, and with mass $m=m_0$. In either case, now apply Lemma \ref{lemma3.2} to obtain from $g$ a new APE metric ${\hat g}$ with mass ${\hat m}<m=m_0$ and ${\hat A}:=A_{\hat g}=0$.

Since we assume that $m_0$ is the mass of the minimizing geon in Conjecture \ref{conjecture1.1}, if that conjecture is true then the above construction must be impossible. Therefore, $L_{g_0}^*$ must have nontrivial kernel. Then there must be a nonconstant function $N$ lying in that kernel. But then $(M,g,N)$ is a solution of the static Einstein equations (\ref{eq2.17}), or equivalently the system (\ref{eq2.15}, \ref{eq2.16}). This proves Theorem \ref{theorem1.2}.(i).

Write equation (\ref{eq2.17}) in the form
\begin{equation}
\label{eq3.20}
\begin{split}
P(N)=&\, 0\ , \\
P(\cdot ):=&\, \hess - g\Delta-\ric\ .
\end{split}
\end{equation}
For some $v$ bounded and sufficiently differentiable, a short calculation yields
\begin{equation}
\label{eq3.21}
x^{-s}P(x^sv)\big \vert_{x=0}=(s+1)\left [ (n-1)dx^2\oplus (n-1-s)h_0\right ]\ .
\end{equation}
Setting the trace of equation (\ref{eq3.20}) equal to zero, we obtain the indicial equation for (\ref{eq2.16}), which is $0=(s+1)(n-s)$, so the indices are $s=-1$ and $s=n$. Then $N$ can diverge as $\frac{1}{x}$ or even $\frac{\log x}{x}$ as $x\to\infty$. The coefficients of the divergent terms are arbitrary functions on conformal infinity: they are ``boundary data''. The assumption that $xN\to 1$ fixes this freedom. (This, together with the static Einstein equations and the APE condition for $g$, are a form of asymptotically anti-de Sitter condition for the spacetime.) Now from Proposition \ref{proposition2.9} with $\lambda=0$ (see also \cite[Section 3.3]{GW}), we obtain that $N$ and $h_0$ have expansions given by \eqref{eq2.20}, and in particular $x^2g$ (for $g$ as in Proposition \ref{proposition2.9}) is a smooth conformal compactification at $x=0$.

There remain two possibilities: either $N$ has zeroes (contained within a bounded region) or it does not.
If $N$ has zeroes contained within a bounded region, it is well-known that the zero set must be a totally geodesic, closed, embedded hypersurface (e.g., \cite[Proposition 2.6]{Corvino}).

Set $n=3$. If the zero set of $N$ is a boundary, or is 2-sided so that cutting along it produces a boundary, then by \cite[Theorem 4]{ACD}\footnote
{This relies on $C^2$ smoothness of the conformal compactification $x^2g$; we in fact have $C^{\infty}$.}
$(M,g,N)$ must be the exterior of a Lemos toroidal black hole \cite{Lemos}. In the non-2-sided case, one achieves this by cutting out the offending component(s) and replacing it (them) with the metric completion. The Lemos toroidal black holes have mass $m>0$ whereas we presume that $m=m_0<0$ since $m_0$ is the mass of the minimizing Horowitz-Myers geon, so the possibility that $N$ has zeroes is excluded.

On the other hand, if $N$ has no zeroes, we can invoke \cite[Theorem 4]{ACD}, which is the adaptation of the underlying result of \cite{Anderson} to the spacetime context. It states that if a static Einstein 4-dimensional spacetime with toroidal conformal infinity obeys $A\equiv R_g+n(n-1)=0$ and if the zero set of $N$ is empty, then the spacetime must be an AdS soliton constructed from $(M,g_0,N)$, where $g_0$ is the corresponding Horowitz-Myers geon metric on a time-symmetric slice $M\simeq {\mathbb R}^2\times S^1$ and $N$ is the corresponding lapse. While there is a countable infinity of Horowitz-Myers geons, the geon of least mass is unique up to isometry (and relabelling of cycles, when there is more than one distinct shortest cycle). This proves conclusion (ii) of the theorem.
\end{proof}

\section{Conclusions}
\setcounter{equation}{0}

\noindent We are unable to obtain the same rigidity for all $n$ that we obtain for $n=3$. The difficulty is the absence of uniqueness theorems for static Einstein metrics in higher dimensions with toroidal conformal infinity. Nonetheless, some limited results are available \cite{GSW1, GSW2, GSW3}. We discuss these results in this section. More important than the results are perhaps the gaps that they leave, since gaps indicate possible mechanisms for the failure of the full Conjecture \ref{conjecture1.1}. One gap concerns black holes. If either the Lemos black holes \cite{Lemos} were known to be the unique static APE black holes with toroidal conformal infinity or simply if all static APE black holes with toroidal conformal infinity were known to have positive mass, then this gap could be closed. However, at present, we must entertain the possibility that neither of these statements is true. The other gap is that, among complete static APEs without horizons, one can show the uniqueness of Horowitz-Myers geons \cite{GSW1, GSW2}, but this requires an additional asymptotic assumption which may be unreasonably strong and may preclude other types of ``geon''.

In this section, we assume that $(M,g,N)$ solves the static Einstein system (\ref{eq2.17}) (equivalently, (\ref{eq2.15}, \ref{eq2.16})), and note that then we always have $R+n(n-1)=0$. We also assume that $(M,g)$ is APE and has flat toroidal conformal infinity and zero Neumann data. From Proposition \ref{proposition2.9} with $\lambda=0$ we have
\begin{equation}
\label{eq4.1}
N(x)=\frac{1}{x}\left (1-\frac{\mu}{2n}x^n\right )+{\mathcal O}(x^n)\ ,\ \mu:=\tr_{\delta}\kappa\ .
\end{equation}

Throughout, we have been using the Fefferman-Graham compactification \cite{FG1, FG2} based on a special defining function yielding Gaussian normal coordinates. In references \cite{GSW1, GSW2, GSW3}, a different conformal compactification is used, one used by Chru\'sciel and Simon \cite{CS} and based on the reciprocal of the lapse. The reciprocal of the lapse differs from a special defining function, as is evident from (\ref{eq4.1}). We begin with a brief comparison of the two compactifications, namely the Fefferman-Graham Gaussian-normal-coordinate compactification which in the present setting is
\begin{equation}
\label{4.2}
{\hat g}:=x^2g = dx^2+\delta +\frac{1}{n}\kappa x^n+{\mathcal O}(x^{n+1})
\end{equation}
and which we have been using all along, and the Chru\'sciel-Simon lapse-based \emph{Fermat metric} compactification
\begin{equation}
\label{eq4.3}
\begin{split}
{\tilde g}:=&\, \frac{1}{x^2N^2} \left ( dx^2+\delta +\frac{1}{n}\kappa x^n+{\mathcal O}(x^{n+1})\right ) \\
=&\, \frac{1}{\left (1-\frac{\mu}{2n} x^n\right )^2} \left ( dx^2+\delta +\frac{1}{n}\kappa x^n+{\mathcal O}(x^{n+1})\right )
\end{split}
\end{equation}
used in \cite{GSW1}, \cite{GSW2}, and \cite{GSW3}, where in the last line we used (\ref{eq4.1}). The main point is that the second fundamental form ${\tilde K}_x$ of constant-lapse (and thus constant-$x$) hypersurfaces computed using ${\tilde g}$ as the ambient metric is given by
\begin{equation}
\label{eq4.4}
{\tilde K}_x=-\left ( \kappa + \mu\delta \right ) x^{n-1}+{\mathcal O}(x^n)\ .
\end{equation}
Note as a check that this gives the mean curvature of these hypersurfaces to be
\begin{equation}
\label{eq4.5}
{\tilde H}_x=-n\mu x^{n-1} +{\mathcal O}(x^n)\ ,
\end{equation}
which agrees with \cite[equation (4.2)]{GSW3}.
Finally, if we compute the second fundamental form ${\hat K}_x$ of the same hypersurfaces using ${\hat g}$ as the ambient metric then we have that
\begin{equation}
\label{eq4.6}
{\hat K}_x= -\kappa x^{n-1}+{\mathcal O}(x^n)\ ,
\end{equation}
so
\begin{equation}
\label{eq4.7}
{\tilde K}_x={\hat K}_x-x^{n-1}\mu \delta +{\mathcal O}(x^n)\ .
\end{equation}

The Fermat compactification obeys a structure theorem which yields some limited uniqueness results.

\begin{theorem}[{\cite[Theorem III.2.1 and Remark III.2.2]{GSW2}}]\label{theorem4.1}
If $(M,g,N)$ is a solution of the static vacuum equations with APE asymptotics, flat toroidal conformal boundary, and negative mass, and if the eigenvalues of ${\tilde K}$ are (positive) semi-definite, then the universal cover of the conformally compactified manifold $({\tilde M},{\tilde g})$, ${\tilde g}:=g/N^2$, splits isometrically as $({\mathbb R}^k\times W, \delta \oplus \sigma )$, where $({\mathbb R}^k,\delta ) \equiv {\mathbb E}^k$ is Euclidean $k$-space and $(W,\sigma)$ is a compact Riemannian manifold with non-empty boundary.
\end{theorem}

A weakness of this theorem is that Horowitz-Myers geons have $\kappa^i{}_j=\diag (-(n-1),1,\dots,1)$, so $\kappa + \mu \delta =\diag (-n,0,\dots,0)$. Then from (\ref{eq4.4}), one does not know whether ${\tilde K}$ for these metrics is positive definite from $\kappa$ alone. With an explicit metric such as the geon, one can just check directly, but in general one must compute the error terms in (\ref{eq4.4}). Thus, one cannot easily pass from Theorem \ref{theorem4.1} to a uniqueness theorem for Horowitz-Myers geons phrased in terms of the Fefferman-Graham compactification. There is such a theorem phrased in terms of the Fermat metric compactification, meaning that it uses the semi-definiteness of ${\tilde K}$. This is \cite[Theorem IV.1]{GSW2}, which requires in addition to the assumptions listed in Theorem \ref{theorem4.1} a topological assumption on bulk and boundary fundamental groups which implies that $\pi_1(M)\simeq {\mathbb Z}^{n-2}$. It also requires that $\mu$ be pointwise negative.

This leaves open the possibility that $(M,g,N)$ may describe a black hole. This is of course not the only open possibility, given in particular the assumption on ${\tilde K}$ in Theorem \ref{theorem4.1}, but it is potentially the most relevant open issue. But \cite[Proposition 1.1]{GSW3} specializes to static black holes and says that, under the same conditions as above except with the assumption that restricts $\pi_1(M)$ is weakened to the condition that $|\pi_1(M)|=\infty$, there are no non-degenerate black holes (that is, no black holes except possibly those whose horizons are double roots of $N$).

These results point to the most likely source of non-uniqueness for static metrics with flat toroidal conformal infinity. Both results apply only when $\pi_1(M)$ has infinite order. In particular, this condition can be violated by static configurations of topologically spherical, totally geodesic boundary components for $(M,g)$; that is, multiple, topologically spherical black holes in static equilibrium. It can be expected that the areas of these horizons would contribute positively to the mass, but their mutual gravitational binding energies would contribute negatively. For suitably chosen configurations (perhaps many small black holes of this sort), the binding energies may exceed the contributions from horizon area, possibly by enough to lead to a violation of the full Horowitz-Myers conjecture if a toroidal horizon does not form before this point is reached. It is interesting that static configurations of this nature cannot exist when the dimension of space is $n=3$, but are not (yet) ruled out in higher dimensions.

\appendix
\section{Proofs of propositions in subsections 2.2 and 2.3}
\setcounter{equation}{0}

\begin{proof}[Proof of Proposition \ref{proposition2.6}]
It is a standard result (e.g., \cite[Proposition 2.3.4]{DGH} that partial evenness is well-defined; i.e., that ${\hat h}_{\hat x}$ is partially even because $h_x$ is. We will repeat part of that proof in order to prove the invariance for $\mu$ for $n$ odd. To begin, we define $\omega(x,\cdot)$ by ${\hat x}=e^{\omega}x$. Since $x$ and ${\hat x}$ are both special defining functions, we have $x^{-2}g^{-1}(dx,dx)={\hat x}^{-2} g^{-1} (d{\hat x}, d{\hat x})=1$ on a neighbourhood of conformal infinity. This yields
\begin{equation}
\label{eqA.1}
2\omega'+x\left ( \omega'^2+h_x^{-1} (d\omega, d\omega ) \right )=0\ .
\end{equation}
Evaluating this at $x=0$, one observes that $\omega'(0)=0$, while differentiating $2k$ times with respect to $x$ and evaluating the result at $x=0$ we get
\begin{equation}
\label{eqA.2}
2\omega^{(2k+1)}=-2k\frac{\partial^{2k-1}}{\partial x^{2k-1}}\bigg \vert_{x=0} \left [ \omega'^2 +h_x^{-1} (d\omega,d\omega)\right ] \ .
\end{equation}
Every term on the right-hand side has either an odd number of derivatives of $\omega$ or an odd number of derivatives of $h^{-1}$, and in either case this odd number is $\le 2k-1$. Since $h_x$ is partially even, the odd derivatives of $h_x^{-1}$ of order $\le 2k-1$ vanish at $x=0$ for $2k-1<n$ if $n$ is odd, and for $2k-1<n-1$ if $n$ is even. Put differently, the odd derivatives of $h_x^{-1}$ of order $<n-1$ vanish at $x=0$, whether $n$ is even or odd. Then by (\ref{eqA.2}) we have that $\omega^{(2k+1)}(0)=0$ whenever $2k+1<n+1$, whether $n$ is even or odd.

Now take $n$ to be odd. Then the first possibly-nonzero odd term in $\omega$ is of order $x^{n+2}$. Since ${\hat x}=e^{\omega}x$, then no \emph{even} term can occur in the expansion of ${\hat x}$ in powers of $x$ below order $x^{n+3}$. It follows that if $x$ is expanded in powers of ${\hat x}$ then the first possibly-nonzero even term is of order ${\hat x}^{n+3}$, and so if $\omega$ is expanded in powers of ${\hat x}$ then the first possibly-nonzero odd term is of order ${\hat x}^{n+2}$.
\begin{equation}
\label{eqA.3}
\begin{split}
\omega =&\, \omega_0 + \text{even\ terms\ }+{\mathcal O}({\hat x}^{n+2})\text{\ for\ }n\text{\ odd}\ , \\
\Rightarrow d\omega =&\, \left ( \text{odd\ terms\ }+{\mathcal O}({\hat x}^{n+1})\right )d{\hat x}
+\left ( \text{even\ terms\ }+{\mathcal O}({\hat x}^{n+2}) \right ) dy\ , \\
\end{split}
\end{equation}
where $dy$ represents differentials of coordinates $y$ on $\partial_{\infty} M$. A brief calculation using
\begin{equation}
\label{eqA.4}
\begin{split}
x=&\, e^{-\omega}{\hat x}\\
dx=&\, e^{-\omega}\left ( d{\hat x} -{\hat x}d\omega \right )
\end{split}
\end{equation}
and
\begin{equation}
\label{eqA.5}
\begin{split}
{\hat h}_{\hat x} =&\, e^{2\omega}\left ( dx^2+h_x\right ) -d{\hat x}^2\\
=&\, e^{2\omega}h_x + {\hat x}^2 d\omega^2 -2{\hat x}d{\hat x}d\omega
\end{split}
\end{equation}
leads to
\begin{equation}
\label{eqA.6}
\begin{split}
{\hat h}_{\hat x} = &\, e^{2\omega}h_x +\text{even\ terms\ } +{\mathcal O}({\hat x}^{n+4}) \\
 = &\, e^{2\omega_0} h_x + \text{even\ terms\ } +{\mathcal O}({\hat x}^{n+2})\ ,
\end{split}
\end{equation}
where ${\mathcal O}({\hat x}^p)$ means that the coefficient of $dy^i dy^j$ is homogeneous in ${\hat x}$ of order $p$. Finally, we have to replace $x$ by ${\hat x}$ in $h_x$, but this cannot produce on odd term below order ${\hat x}^{n+2}$.
\begin{equation}
\label{eqA.7}
{\hat h}_{\hat x} = e^{2\omega_0}\left ( h_{(0)}+x^2h_{(2)}+\dots+x^{n-1}h_{(n-1)} +\frac{x^n}{n}\kappa\right ) + \text{even\ terms\ } +{\mathcal O}({\hat x}^{n+2})\ .
\end{equation}
We can re-express each power of $x$ using $x=e^{-\omega}{\hat x}=e^{-\omega_0} x\left ( 1 + \text{even\ terms\ }\right ) +{\mathcal O}({\hat x}^{n+3})$. We get
\begin{equation}
\label{eqA.8}
\begin{split}
{\hat h}_{\hat x} = &\, e^{2\omega_0} h_{(0)}+{\hat x}^2 {\check h}_{(2)}+\dots+e^{-(n-3)\omega_0}{\hat x}^{n-1}{\check h}_{(n-1)}+e^{-(n-2)\omega_0}\frac{{\hat x}^n}{n}\kappa\\
&\, + \text{even\ terms\ } +{\mathcal O}({\hat x}^{n+2})\ ,
\end{split}
\end{equation}
where in general ${\check h}_{(2k)}$ differs from $h_{(2k)}$, for $k=2,4,\dots$, but ${\check h}_{(2)}=h_{(2)}$ and, more importantly here, the order $x^n$ term is unchanged. If we write the expansion of ${\hat h}_{\hat x}$ by
\begin{equation}
\label{eqA.9}
{\hat h}_{\hat x} = {\hat h}_{(0)}+{\hat x}^2{\hat h}_{(2)}+\dots+\frac{{\hat x}^{n-1}}{(n-1)}{\hat \theta}+\frac{{\hat x}^n}{n}{\hat \kappa}+{\mathcal O}({\hat x}^{n+1})\ ,
\end{equation}
then we read off that
\begin{equation}
\label{eqA.10}
{\hat \mu}:=\tr_{{\hat h}_{(0)}}{\hat \kappa}=e^{-n\omega_0}\tr_{h_{(0)}}\kappa=:e^{-n\omega_0}\mu\ .
\end{equation}
\end{proof}

\begin{proof}[Proof of Proposition \ref{proposition2.7} and Remark \ref{remark2.8}]
If we multiply equation (\ref{eq2.5}) by $-2x$ we obtain
\begin{equation}
\begin{split}
\label{eqA.11}
-2xE^{\perp}=&\, xh_x''-(n-2)h_x'-h_{x} \tr_{h_x}h_x'-x h_x'\cdot h_x^{-1}\cdot h_x'\\
&\, +\frac12xh_x'\tr_{h_x}h_x'-2x\ric(h_x)\ .
\end{split}
\end{equation}
If we differentiate this expression $n-1$ times with respect to $x$ and set $x=0$, we get
\begin{equation}
\begin{split}
\label{eqA.12}
-2(n-1)\frac{\partial^{n-2}}{\partial x^{n-2}}\bigg \vert_{x=0} E^{\perp}=&\, h^{(n)}(0) - \frac{\partial^{n-1}}{\partial x^{n-1}}\bigg \vert_{x=0} \left [ h(x) \tr_{h(x)}h'(x)\right ]\\
&\, +\frac12 (n-1)\frac{\partial^{n-2}}{\partial x^{n-2}}\bigg \vert_{x=0} \bigg [ h'(x) \tr_{h(x)}h'(x)\\ &\, -2h_x'\cdot h_x^{-1}\cdot h_x'-4\ric(h(x))\bigg ]\ .
\end{split}
\end{equation}
Since $E\in {\mathcal O}(x^{n+1})$, the components of $E$ in our basis are ${\mathcal O}(x^{n-1})$ and therefore the left-hand side of (\ref{eqA.12}) vanishes. On the right-hand side, we first note that $h^{(n)}(0)=(n-1)!\kappa$. For any partially even metric (\ref{eq2.12})), we have $h'(0)=0$. Therefore, the only terms on the right-hand side that are not zero are those that do not have a factor $h'(0)$. Such terms either contain $\kappa$ or contain only $h(0),\dots,h^{(n-2)}(0)$. This observation yields
\begin{equation}
\label{eqA.13}
0=\kappa-h_{(0)}\tr_{h_{(0)}}\kappa+F(h_{(0)},h_{(2)},\dots ,h^{(n-2)}(0))\ ,
\end{equation}
where, as in (\ref{eq2.14}), we use the notation $h_{(k)}:=\frac{1}{k!}h^{(k)}(0)$. But we can then manipulate this equation to remove the $h_{(0)}\tr_{h_{(0)}}\kappa$ term and write that
\begin{equation}
\label{eqA.14}
\kappa=G(h_{(0)},h_{(2)},\dots ,h^{(n-2)}(0))
\end{equation}
for some $G$. This shows that $\kappa$ does not depend on $h^{(n-1)}(0)$.

Then $\tr_{h_0}\kappa$ is independent of the Neumann data, so the mass relative to $h_0$ of any two metrics $(M,g_1)$, $(M,g_2)$ with $E_{g_1},E_{g_2}\in {\mathcal O}(x^{n+1})$, both of have $(\partial_{\infty}M,[h_0])$ as their conformal boundary, will be equal. But when $\ric(h_0)=(n-2)\lambda h_0$, there is always one such metric given by (\ref{eq2.9}) and it has zero mass since it has $\kappa=0$. Hence any other metric $(M,g)$ with $E_{g}\in {\mathcal O}(x^{n+1})$ which has $(\partial_{\infty}M,[h_0])$ as its conformal boundary will have zero mass (where that mass is defined relative to the Einstein metric $h_0$ if $n$ is even). \end{proof}

\begin{proof}[Proof of Proposition \ref{proposition2.9}]
Since $(M,g,N)$ is a solution of the static Einstein equations, $N^2d\tau^2 \oplus g$ is Poincar\'e-Einstein and has $\frac{\partial}{\partial \tau}$ as a Killing field. Then a special defining function $x$ for the APE metric $g$ on $M$ is also a special defining function for the Poincar\'e-Einstein metric $N^2d\tau^2 \oplus g$.

Let us recall \cite[Theorem 2.3.1]{DGH} (see also \cite[Theorem 4.8]{FG2}). This theorem says that since $N^2d\tau^2 \oplus g$ is Poincar\'e-Einstein metric on an $(n+1)$-manifold, terms up to order $x^{n-1}$ (inclusive) are uniquely determined by the boundary metric $d\tau^2+h_0$, as is the trace of the order $x^n$ term. The ambient obstruction tensor is also uniquely determined by $N^2 d\tau^2 \oplus g$ (and vanishes when $n$ is odd).

Now the projection orthogonal to $\frac{\partial}{\partial x}$ of the Einstein tensor of the spacetime metric is given by equation \eqref{eqA.11} with $n$ replaced by $n=1$ and $E^{\perp}=0$. We will also replace $h_x$ by ${\hat h}_x=x^2 N^2(x)d\tau^2\oplus h_x$ in \eqref{eqA.11}, with ${\hat h}_0 := d\tau^2 \oplus h_0$. Then we have
\begin{equation}
\label{eqA.15}
\begin{split}
0=&\, x{\hat h}_x''-(n-1){\hat h}_x'-{\hat h}_{x} \tr_{{\hat h}_x}{\hat h}_x'-x {\hat h}_x'\cdot {\hat h}_x^{-1}\cdot {\hat h}_x'\\
&\, +\frac12x{\hat h}_x'\tr_{{\hat h}_x}{\hat h}_x'-2x\ric({\hat h}_x)\ .
\end{split}
\end{equation}
Noting that $\ric_{{\hat h}_x}=0\cdot d\tau^2 \oplus \ric_{h_x}= 0\cdot d\tau^2 \oplus (n-2)\lambda h_0$ for $h_0$ an Einstein metric $\ric_{h_0}=(n-2)\lambda$, $\lambda\in \{ -1,0,1\}$, one can easily verify that
\begin{equation}
\label{eqA.16}
\begin{split}
g=&\, \frac{1}{x^2}\left ( dx^2 + {\hat h}_x\right )= \frac{1}{x^2}\left ( dx^2 + x^2N^2(x)d\tau^2 + h_x\right )\ , \\
h_x=&\, \left ( 1-\frac{\lambda x^2}{4} \right )^2 h_0\ , \\
N(x)=&\, \frac{1}{x}+\frac{\lambda x}{4}
\end{split}
\end{equation}
solves \eqref{eqA.15}. Though we have only checked that $E^{\perp}$ is zero, the full tensor $E$ vanishes. Therefore $g$ in \eqref{eqA.16} is Poincar\'e-Einstein and induces $[{\hat h}_0] =[d\tau^2\oplus h_0]$ as its conformal boundary. Furthermore, there is no log term so the obstruction tensor vanishes. But since the obstruction depends only on the metric on the conformal boundary, it vanishes for all Poincar\'e-Einstein metrics $g$ with this conformal boundary, and so $x^2g$ extends smoothly to $x=0$ for any such Poincar\'e-Einstein metric $g$ and related special defining function $x$.

As well, for $g$ as in \eqref{eqA.16}, let ${\bar g}$ denote a distinct Poincar\'e-Einstein metric with the same conformal boundary. Expressed in normal form, ${\bar g}$ can differ from $g$ only in the Neumann data term and beyond. By \emph{Neumann data}, we mean the lowest order components of ${\hat h}_x$ not determined by \eqref{eqA.15} in terms of ${\hat h}_0$. It's well-known that the Neumann data comprise the tracefree part of ${\hat h}_x^{(n)}\big \vert_{x=0}$. (To see this, differentiate \eqref{eqA.15} $k-1$ times, then set $x=0$, and notice what happens when $k=n$ versus when $k\neq n$.) In particular, ${\bar g}$ must agree with $g$ as given by \eqref{eqA.16} to order $x^{n-1}$ inclusive, and therefore must have normal form
\begin{equation}
\label{eqA.17}
\begin{split}
{\bar g}=&\, \frac{1}{x^2}\left ( dx^2+ {\hat h}_x\right )= \frac{1}{x^2}\left ( dx^2 + x^2N^2(x)d\tau^2 + h_x\right )\ , \\
h_x=&\, \left ( 1 -\frac14 \lambda x^2\right )^2 h_0 +\frac{1}{n}\kappa x^n +{\mathcal O}(x^{n+1})\\
N=&\, \frac{1}{x} \left ( 1+ \frac14 \lambda x^2\right )+\frac{a}{2}x^{n-1} +{\mathcal O}(x^n)\ .
\end{split}
\end{equation}
Furthermore, it's also well-known (and can be seen by $x$-differentiating \eqref{eqA.15} $n-1$ times at $x=0$) that the trace of the $n^{\rm th}$ $x$-derivative of ${\hat h}$ must vanish at $x=0$, which implies that
\begin{equation}
\label{eqA.18}
\begin{split}
0=&\, a +\frac{1}{n}\tr_{h_0}\kappa\\
\implies a=&\, -\frac{1}{n}\tr_{h_0} \kappa \equiv -\frac{\mu}{n}\ .
\end{split}
\end{equation}
\end{proof}


\begin{thebibliography}{99}
\bibitem{AD} LF Abbott and S Deser, \emph{Stability of gravity with a cosmological constant}, Nuc Phys B195 (1982) 76--96.
\bibitem{Anderson} MT Anderson, \emph{Boundary regularity, uniqueness and non-uniqueness for AH Einstein metrics on $4$-manifolds}, Adv Math 179 (2003) 205--249.
\bibitem{ACD} MT Anderson, PT Chru\'sciel, and E Delay, \emph{Non-trivial, geodesically complete, vacuum space-times with a negative cosmological constant}, JHEP 10 (2002) 063.
\bibitem{ACG} L Andersson, M Cai, and GJ Galloway, \emph{Rigidity and positivity of mass for asymptotically hyperbolic manifolds}, Ann Henri Poincar\'e 9 (2008) 1-–33.
\bibitem{AnderssonDahl} L Andersson and M Dahl, \emph{Scalar curvature rigidity for asymptotically locally hyperbolic manifolds}, Ann Global Anal Geom 16 (1998) 1-–27.
\bibitem{AshtekarMagnon} A Ashtekar and A Magnon, \emph{Asymptotically anti-de Sitter space-times}, Class Quantum Gravit 1 (1984) L39--L44.
\bibitem{AM} P Aviles and RC McOwen, \emph{Curvature deformation to constant negative scalar curvature on noncompact Riemannian manifolds}, J Differential Geometry 27 (1988) 225--239.
\bibitem{Bahuaud} E Bahuaud, \emph{Ricci flow of conformally compact metrics}, Ann Inst H Poincar\'e: Analyse Non Lin\'eaire 28 (2011) 813--835.
\bibitem{BMW1} E Bahuaud, R Mazzeo, and E Woolgar, \emph{Renormalized volume and the evolution of APEs}, preprint [arxiv:1307.4788].
\bibitem{BMW2} E Bahuaud, R Mazzeo, and E Woolgar, \emph{Ricci flow and volume renormalizability}, manuscript in progress.
\bibitem{BW} T Balehowsky and E Woolgar, \emph{The Ricci flow of asymptotically hyperbolic mass and applications}, J Math Phys 53 (2012) 072501.
\bibitem{BP} D Brill and H Pfister, \emph{States of negative total energy in Kaluza-Klein theory}, Phys Lett B228 (1989) 359--362.
\bibitem{CH} PT Chru\'sciel and M Herzlich, \emph{The mass of asymptotically hyperbolic Riemannian manifolds}, Pacific J Math 212 (2003) 231-–264
\bibitem{CS} PT Chru\'sciel and W Simon, \emph{Towards the classification of static vacuum spacetimes with negative cosmological constant}, J Math Phys 42 (2001) 1779--1817.
\bibitem{CM} NR Constable and RC Myers, \emph{Spin-two glueballs, positive energy theorems and the AdS/CFT correspondence}, JHEP 9910 (1999) 037.
\bibitem{Corvino} J Corvino, \emph{Scalar curvature deformation and a glueing construction for the Einstein constraint equations}, Commun Math Phys 214 (2000) 137--189.
\bibitem{DGH} Z Djadli, C Guillarmou, and M Herzlich, \emph{Op\'erateurs g\'eom\'etriques, invariants conformes et vari\'et\'es asymptotiquement hyperboliques}, Panoramas et synth\`eses 26 (Soci\'et\'e math\'ematique de France, Luminy, 2008).
\bibitem{FG1} C Fefferman and CR Graham, \emph{Conformal invariants}, in \emph{\'Elie Cartan et les math\'ematiques d'aujourd'hui}, Ast\'erisque (num\'ero hors s\'erie, 1985) 95--116.
\bibitem{FG2} C Fefferman and CR Graham, \emph{The ambient metric}, Princeton annals of mathematical studies 178 (Princeton, 2012).
\bibitem{GM} GJ Galloway and P Miao, \emph{Variational and rigidity properties of static potentials}, Commun Anal Geom (to appear).
\bibitem{GSW1} GJ Galloway, S Surya, and E Woolgar, \emph{A uniqueness theorem for the anti-de Sitter soliton}, Phys Rev Lett 88 (2002) 101102.
\bibitem{GSW2} GJ Galloway, S Surya, and E Woolgar, \emph{On the geometry and mass of static, asymptotically AdS spacetimes, and the uniqueness of the AdS soliton}, Commun Math Phys 241 (2003) 1--25.
\bibitem{GSW3} GJ Galloway, S Surya, and E Woolgar, \emph{Non-existence of black holes in certain $Λ<0$ spacetimes}, Class Quantum Gravit 20 (2003) 1635--1648.
\bibitem{GW} GJ Galloway and E Woolgar, \emph{On static Poincar\'e-Einstein metrics} JHEP 06(2015) 051.
\bibitem{Geroch} R Geroch, {\it Energy extraction}, Ann NY Acad Sci 224 (1973) 108--117.
\bibitem{Gibbons} GW Gibbons, \emph{Some Comments on Gravitational Entropy and the Inverse Mean Curvature Flow}, Class Quantum Gravit 16 (1999) 1677--1687.
\bibitem{GHW} GW Gibbons, CM Hull, and NP Warner, \emph{The stability of gauged supergravity}, Nuc Phys B218 (1983) 173--190.
\bibitem{HP} SW Hawking and DN Page, \emph{Thermodynamics of black holes in anti-de Sitter space}, Commun Math Phys 87 (1983) 577--588.
\bibitem{HS} M Henningson and K Skenderis, \emph{The holographic Weyl anomaly}, JHEP 9807 (1998) 023.
\bibitem{HM} GT Horowitz and RC Myers, \emph{The AdS/CFT correspondence and a new positive energy conjecture for general relativity}, Phys Rev D59 (1999) 026005.
\bibitem{HI} G Huisken and T Ilmanen, {\it The Riemannian Penrose Inequality}, Int Math Res Not 20 (1997) 1045--1058; {\it The Inverse Mean Curvature Flow and the Riemannian Penrose Inequality}, J Diff Geom 59 (2001) 353--437.
\bibitem{LN} DA Lee and A Neves, \emph{The Penrose inequality for asymptotically locally hyperbolic spaces with nonpositive mass}, Commun Math Phys 339 (2015) 327--352.
\bibitem{Lemos} JPS Lemos, \emph{Cylindrical Black Hole in General Relativity}, Phys Lett B353 (1995) 46--51.
\bibitem{Maldacena}  J Maldacena, \emph{The large $N$ limit of superconformal field theories and supergravity}, Adv Theor Math Phys 2 (1998) 231--252.
\bibitem{MinOo} M Min-Oo, \emph{Scalar curvature rigidity of asymptotically hyperbolic spin manifolds}, Math Ann 285 (1989) 527--539.
\bibitem{Page} DN Page, \emph{Phase transitions for gauge theories on tori from the AdS/CFT correspondence}, JHEP 0809:037 (2008)
\bibitem{SY} R Schoen and S-T Yau, {\it On the proof of the positive mass conjecture in general relativity}, Commun Math Phys 65 (1979) 45--76.
\bibitem{Wang} X Wang, \emph{The mass of asymptotically hyperbolic manifolds}, J Differential Geometry 57 (2001) 273-–299.
\bibitem{Witten1} E Witten, \emph{A new proof of the positive energy theorem}, Commun Math Phys 80 (1981) 381--402.
\bibitem{Witten2} E Witten, \emph{Instability of the Kaluza-Klein vacuum}, Nucl Phys B195 (1982) 481--492.
\bibitem{Witten3} E Witten, \emph{Anti-de Sitter space and holography} Adv Theor Math Phys 2 (1998) 253--291.
\bibitem{Woolgar} E Woolgar, \emph{The Positivity of Energy for Asymptotically Anti-de Sitter Spacetimes}, Class Quantum Gravit 11 (1994) 1881--1900
\end{thebibliography}
\end{document}